\documentclass[journal,9pt]{IEEEtran}
\usepackage[T1]{fontenc}
\usepackage[francais,english]{babel}
\usepackage{a4,latexsym,amssymb,amsfonts,pgf}
\usepackage{enumerate}
\usepackage{amsthm}
\usepackage{amsmath}
\usepackage[all,2cell]{xy}
\usepackage{amsmath} 
\usepackage{amsfonts} 
\usepackage{amsthm} 
\usepackage{epsfig} 
\usepackage{graphicx} 
\usepackage{multirow} 
\usepackage{verbatim} 
\usepackage{rotating} 
\usepackage{amscd}
\usepackage{pstricks}
\usepackage{pst-tree}
\usepackage{pst-node}
\usepackage{dsfont}
\usepackage{lscape}
\makeatletter
\def\markboth#1#2{\def\leftmark{\@IEEEcompsoconly{\sffamily}\MakeUppercase{\protect#1}}%
\def\rightmark{\@IEEEcompsoconly{\sffamily}\MakeUppercase{\protect#2}}}
\makeatother

\ifCLASSINFOpdf

\else

\fi

\newtheorem{them}{Th\'eor\`eme}
\newtheorem{Def}{D\'efinition}
\newtheorem{propo}{Proposition}
\newtheorem{coro}{Corollaire}

\newtheorem{exs}{Exemples}

\newtheorem{lem}{Lemme}
\newtheorem{rem}{Remarque}
\newtheorem{rems}{Remarques}
\newcommand{\bexs}{\begin{exs}}
\newcommand{\eexs}{\end{exs}}
\newcommand{\bt}{\begin{them}}
\newcommand{\et}{\end{them}}
\newcommand{\bd}{\begin{Def}}
\newcommand{\ed}{\end{Def}}
\newcommand{\bl}{\begin{lem}}
\newcommand{\el}{\end{lem}}
\newcommand{\bp}{\begin{propo}}
\newcommand{\ep}{\end{propo}}
\newcommand{\brem}{\begin{rem}}
\newcommand{\erem}{\end{rem}}
\newcommand{\brems}{\begin{rems}}
\newcommand{\erems}{\end{rems}}
\newcommand{\bc}{\begin{coro}}
\newcommand{\ec}{\end{coro}}
\newcommand{\bpr}{\begin{proof}}
\newcommand{\epr}{\end{proof}}
\newcommand{\K}{\mathbb{Q}(\sqrt{d},i)}
\def\NN{\mathbb{N}}

\def\QQ{\mathbb{Q}}

\def\ZZ{\mathbb{Z}}
\def\OO{\mathcal{O}}
\def\kk{\mathds{k}}
\def\KK{\mathds{K}}
\def\k{\mathds{k}^{(*)}}

\begin{document}

\title{Sur la capitulation des 2-classes d'idéaux du corps $\QQ(\sqrt{2p_1p_2},i)$}

\author{Abdelmalek~Azizi, Abdelkader~Zekhnini,~\IEEEmembership{FS,~Oujda}
et Mohammed~Taous,~\IEEEmembership{FST,~Errachida}}

\maketitle

\begin{abstract}
  Let $p_1$ and $p_2$ be different  primes  such that $p_1\equiv p_2\equiv1 \pmod4$ and at least two of the three  elements $\{(\frac{2}{p_1}), (\frac{2}{p_2}), (\frac{p_1}{p_2})\}$ are equal to -1. Put $i=\sqrt{-1}$, $d=2p_1p_2$ and $k =Q(\sqrt{d}, i)$. Let  $k_2^{(1)}$ be the Hilbert 2-class field of $k$ and $k^{(*)}=Q(\sqrt{p_1},\sqrt{p_2},\sqrt 2, i)$ be its genus field. Let $C_{k,2}$ denote The 2-part of the class group of $k$. The  unramified abelian extensions of $k$  are $K_1=k(\sqrt{p_1})$, $K_2=k(\sqrt{p_2})$, $K_3=k(\sqrt{2})$ and $k^{(*)}$. Our goal is to study the   capitulation  problem of the 2-classes of $k$ in these four extensions.
\end{abstract}
\selectlanguage{francais}
\begin{abstract}
 Soient $p_1$ et $p_2$ deux nombres premiers tels que $p_1\equiv p_2\equiv1 \pmod
 4$ et  au moins deux des éléments \{$(\frac{p_1}{p_2})$, $(\frac{2}{p_1})$, $(\frac{2}{p_2})$\} valent -1. Soient $i=\sqrt{-1}$, $d=2p_1p_2$, $\kk=\K$, $\kk_2^{(1)}$ le 2-corps de classes de Hilbert de
 $\kk$ et $\k$ le corps de genres de $\kk$. Notons par $\mathbf{C}_{\mathds{\kk},2}$ la 2-partie du groupe de
 classes de $\kk$. Les extensions abéliennes sur $\QQ$ et non ramifiées sur  $\kk$ sont $\KK_1=\kk(\sqrt{p_1})$, $\KK_2=\kk(\sqrt{p_2})$, $\KK_3=\kk(\sqrt{2})$ et $\k=\QQ(\sqrt 2, \sqrt{p_1}, \sqrt{p_2}, i)$. Dans ce papier on
 s'intéresse à étudier le problème de capitulation des 2-classes de $\kk$ dans ces quatre extensions.
\end{abstract}

\begin{IEEEkeywords}
$2$-groupe de classes, corps de classes de Hilbert, corps de genre, capitulation.
\end{IEEEkeywords}

\IEEEpeerreviewmaketitle
\hfill Oujda

\hfill janvier 26, 2012

\section{Introduction}
\IEEEPARstart{S}{o}ient
 $k$ un corps de nombres de degré fini sur $\QQ$, $K$ une extension non ramifiée de
 $k$ et $p$ un nombre premier. La recherche des idéaux de $k$ qui capitulent dans $K$ a
 été l'objet d'etude d'un grand nombre de mathématiciens. Dans le cas où $K$ est égal
 au corps de classes de Hilbert $k^{(1)}$ de $k$, D. Hilbert avait conjecturé que toutes
 les classes de k capitulent dans $k^{(1)}$ (théorème de l'idéal principal). Ce
 théorème a été démontré par E. Artin et Ph. Furtw\"angler. Le cas où $K/k$ est une
 extension cyclique et $[K:k]=p$, un nombre premier, a été étudié par Hilbert dans son
 théorème 94 (voir par exemple \cite {Mi-89}) qui affirme qu'il y a au moins une
 classe non triviale dans $k$ qui capitule dans $K$. De plus, Hilbert avait trouvé le résultat suivant: si $Gal(K/k)=\langle\sigma\rangle$, $N$ la norme de $K/k$, $U_k$ le groupe des unités de $k$, $U_K$ le groupe des unités de $K$ et $U^*$ le sous-groupe des unités de $U_K$ dont la norme est égale à 1. Alors le groupe des classes de $k$ qui capitulent dans $K$ est isomorphe au groupe quotient $U^*/U_K^{1-\sigma}=H^1(U_K)$, le groupe cohomologique de $U_K$ de dimension 1.\par
 A l'aide de ce résultat et d'autre résultats, on montre le théorème suivant:
 \begin{them}\label{17}
 Soit $K/k$ une extension cyclique de degré un nombre premier, alors le nombre des classes qui capitulent dans $K/k$ est égale à: $[K:k][E_k:N(E_K)]$
 \end{them}
 Soit $d=2p_1p_2$, avec $p_1$ et $p_2$ deux nombres premiers
tels que $p_1\equiv p_2\equiv 1\pmod 4$  et  au moins deux des éléments \{$(\frac{p_1}{p_2})$, $(\frac{2}{p_1})$, $(\frac{2}{p_2})$\} valent -1 et $\mathds{k} = \K$. On note par $\mathbf{C}_{\mathds{k},2}$ le $2$-groupe de classes de $\mathds{k}$, $\mathds{k}^{(*)}$ le corps de genres de $\mathds{k}$ et
$\mathds{k}^{(1)}_2$ le corps de classes de Hilbert de $\mathds{k}$. Notre but est
d'étudier le problème de capitulation dans les extensions quadratiques abéliennes sur $\QQ$ non ramifiées sur $\kk$. \par
 D'après \cite{AzTa08}, $\mathbf{C}_{\mathds{\kk},2}$ est de type $(2, 2, 2)$ et
 $[\kk^{(*)}:\kk]=4$. Alors le rang du $2$-groupe de classes est égal à 3, ainsi $\kk$
 admet sept extensions quadratiques non ramifiées dans $\mathds{k}^{(1)}_2$.
 Notons $\KK_1=\kk(\sqrt{p_1})$, $\KK_2=\kk(\sqrt{p_2})$ et $\KK_3=\kk(\sqrt{2})$.
 On définit d'abord les générateurs de $\mathbf{C}_{\mathds{k},2}$.
\section{Les générateurs de $\mathbf{C}_{\mathds{k},2}$.}
  Soient $a$ un entier naturel sans facteurs carrés et $Q$ l'indice d'unités du corps $\QQ(\sqrt a,i)$, alors on a :
\begin{lem}[\cite{AzTa08}]\label{3}
Si l'une des conditions suivantes est vérifiée, alors $Q=1.$
\begin{enumerate}
  \item $a$ est congru à 1 modulo 4.
  \item Il existe un entier impair $a'$ qui divise $a$ tel que $a'\equiv5\pmod 8.$
\end{enumerate}
\end{lem}
\begin{lem}[\cite{AzTa08}]\label{4}
Si $d=2p_1p_2$ avec $p_1$ et $p_2$ sont deux nombres  premiers tels
que $p_1\equiv p_2\equiv1\pmod 4$ et  au moins deux des éléments \{$(\frac{p_1}{p_2})$, $(\frac{2}{p_1})$, $(\frac{2}{p_2})$\} valent -1, alors la norme de l'unité fondamentale de $\QQ(\sqrt{2p_1p_2})$ est égale à -1.
\end{lem}
\begin{propo}\label{8}
Soient $d$ un entier naturel sans facteurs carrés, $k=\QQ(\sqrt d,i)$, $a+ib$ un élément de $\QQ(i)$ tel que $\sqrt{a^2+b^2}\not\in\QQ(\sqrt d)$ et $\mathcal{H}$ un idéal de $k$ tel que $\mathcal{H}^2=(a+ib)$. Alors  $\mathcal{H}$ est d'ordre 2 dans $k$.
\end{propo}
\begin{proof}
Soient $a+ib$ un élément de $\QQ(i)$ et $\mathcal{H}$ un idéal de $k$ tels que $\sqrt{a^2+b^2}\not\in\QQ(\sqrt d)$ et $\mathcal{H}^2=(a+ib)$. On suppose que $\mathcal{H}$ est principal, alors il existe un $\alpha \in \kk$ et une unité $\varepsilon$ de $\kk$ tels que
\begin{equation} \label{7}\alpha^2=(a+ib)\varepsilon.\end{equation}
Soit $\varepsilon_d$ l'unité fondamentale de $\QQ(\sqrt d)$, alors un SFU de $k$ est $\{\varepsilon_d\}$ ou $\{\sqrt{i\varepsilon_d}\}$; dans ce dernier cas $\varepsilon_d$ est de norme 1. Donc on se ramène aux cas $\varepsilon\in\{\pm1, \pm i, \varepsilon_d, i\varepsilon_d\}$ ou $\varepsilon\in\{\pm1, \pm i, \varepsilon_d, i\varepsilon_d, i\sqrt{i\varepsilon_d}, \sqrt{i\varepsilon_d}\}$.
\begin{itemize}
  \item Si $\varepsilon\in\{\pm1, \pm i, \varepsilon_d, i\varepsilon_d\}$, alors en appliquant la norme $N_{\kk/\QQ(\sqrt d)}$ à l'équation (\ref{7}), on trouve que $a^2+b^2$ est un carré dans $\QQ(\sqrt d)$, ce qui est n'est pas le cas.
  \item Si  $\varepsilon=\sqrt{i\varepsilon_d}$ ou  $\varepsilon=i\sqrt{i\varepsilon_d}$, alors  la norme $N_{\kk/\QQ(i)}$ appliqué à l'équation (\ref{7}), nous donne que $i$ est un carré dans $\QQ(i)$, ce qui est absurde.
\end{itemize}
Donc $\mathcal{H}$ est d'ordre 2 dans $k$.
\end{proof}
On procède comme dans la proposition (8) de l'article \cite{Az-03}, pour  prouver la proposition suivante:
\begin{propo} \label{6}
Soient $d$ un entier composé, pair, sans facteurs carrés et produit d'au moins trois nombres premiers, $k=\QQ(\sqrt d,i)$, $p$ un nombre premier impair et $\mathcal{H}$ un idéal de $k$ tel que $\mathcal{H}^2=(p)$. Notons par $\varepsilon_d=x+y\sqrt d$ et $\mathbf{C}_{k,2}$, l'unité fondamentale de $\QQ(\sqrt d)$ et le 2-groupe de classes de $k$ respectivement. Alors on a:\\
 \indent (1) Si la norme de $\varepsilon_d$ est égale $-1$, alors $\mathcal{H}$ est d'ordre 2 dans $\mathbf{C}_{k,2}$.\\
\indent (2) Si la norme $\varepsilon_d$ est égale à $1$, on a:\\
  (i) Si $\{\varepsilon_d\}$ est SFU de $k$, alors $\mathcal{H}$ est principal si et seulement si $2p(x\pm1)$ ou $p(x\pm1)$  est un carré dans $\NN$.\\
  (ii) Sinon, $\mathcal{H}$ est d'ordre 2 dans $\mathbf{C}_{k,2}$.
\end{propo}
On peut alors citer le résultat:
\begin{them}\label{10}
Soient $\kk=\K$ où $d=2p_1p_2$ avec $p_1$ et $p_2$ sont deux nombres  premiers tels
que $p_1\equiv p_2\equiv1\pmod 4$. Posons $p_1=\pi_1\pi_2$,
où $\pi_1$ et $\pi_2$ sont dans $\ZZ[i]$, notons par $\mathcal{H}_0$,
$\mathcal{H}_1$ et $\mathcal{H}_2$ les idéaux de $\kk$ au dessus de $1+i$, $\pi_1$ et $\pi_2$ respectivement. Si la norme de l'unité fondamentale de $\QQ(\sqrt d)$ est égale à -1, alors le groupe engendré par les
classes de $\mathcal{H}_0$, $\mathcal{H}_1$ et $\mathcal{H}_2$ est de type $(2, 2, 2)$.
\end{them}
\begin{proof}
 Les nombres $\pi_1$et $\pi_2$   sont des premiers ramifiés dans $\kk/\QQ(i)$, alors ils existent des idéaux premiers $\mathcal{H}_1$ et $\mathcal{H}_2$ de $\kk$ tels que: $\pi_j\OO_\kk=(\pi_j)=\mathcal{H}_j^2$,  ($j\in\{1,2\}$),
où $\OO_{\kk}$ est l'anneau des entiers du corps $\kk$. D'autre part $2$ est totalement ramifié dans $\kk$, donc il existe  un idéal premier $\mathcal{H}_0$ de $\kk$ tel que $\mathcal{H}_0^2=(1+i)\OO_\kk$.\\
\indent Comme  la norme de l'unité fondamentale de $\QQ(\sqrt d)$ est égale à -1, alors $Q$, l'indice d'unités de $\kk$, est égal à 1. On a  $\mathcal{H}_0^2=(1+i)$, $\mathcal{H}_1^2=(e+2if)$ et $\mathcal{H}_2^2=(e-2if)$, or $\sqrt2\not\in\QQ(\sqrt d)$ et $\sqrt{e^2+(\pm2f)^2}=\sqrt p_1\not\in\QQ(\sqrt d)$, alors d'après la proposition (\ref{8}) $\mathcal{H}_0$, $\mathcal{H}_1$ et $\mathcal{H}_2$ sont d'ordre $2$ dans $\kk$. De plus $(\mathcal{H}_1\mathcal{H}_2)^2=(p_1)$, alors d'après la proposition (\ref{6}) l'ideal $\mathcal{H}_1\mathcal{H}_2$ est d'ordre $2$ dans $\kk$, d'autre part on a $(\mathcal{H}_0 \mathcal{H}_1)^2=((e-2f)+i(e+2f))$ et $(\mathcal{H}_0 \mathcal{H}_2)^2=((e+2f)+i(e-2f))$ et $\sqrt{(e-2f)^2+(e+2f)^2}=\sqrt{2p_1}\not\in\QQ(\sqrt d)$, donc d'après la proposition (\ref{8}) $\mathcal{H}_0 \mathcal{H}_1$ et $\mathcal{H}_0 \mathcal{H}_2$ sont d'ordre $2$ dans $\kk$; enfin $(\mathcal{H}_0 \mathcal{H}_1 \mathcal{H}_2)^2=((1+i)\pi_1\pi_2)=(p_1+ip_1)$ et comme $\sqrt{p_1^2+p_1^2}=p_1\sqrt2\not\in\QQ(\sqrt d)$, alors $\mathcal{H}_0 \mathcal{H}_1 \mathcal{H}_2$ est d'ordre $2$ dans $\kk$. Et ceci montre que le sous-groupe du 2-groupe de classes de $\kk$, engendré par les classes de $\mathcal{H}_0$, $\mathcal{H}_1$ et $\mathcal{H}_2$, est de type $(2, 2, 2)$.
\end{proof}
\begin{coro}
Gardons les hypothèses du théorème précèdent et notons par $\mathbf{C}_{\mathds{\kk},2}$ la 2-partie du groupe de
 classes de $\kk$. Si $\mathbf{C}_{\mathds{\kk},2}$ est de type $(2, 2, 2)$, alors au moins deux des éléments \{$(\frac{p_1}{p_2})$, $(\frac{2}{p_1})$, $(\frac{2}{p_2})$\} valent -1. Dans ce cas on a: $\mathbf{C}_{\mathds{\kk},2}=\langle[\mathcal{H}_0], [\mathcal{H}_1],  [\mathcal{H}_2]\rangle$
\end{coro}
\begin{proof}
D'après le théorème (\ref{10}) on a $\langle[\mathcal{H}_0], [\mathcal{H}_1],  [\mathcal{H}_2]\rangle$ est de type $(2, 2, 2)$, par suite $\mathbf{C}_{\mathds{\kk},2}=\langle[\mathcal{H}_0], [\mathcal{H}_1],  [\mathcal{H}_2]\rangle$, or d'après \cite{AzTa08} $\mathbf{C}_{\mathds{\kk},2}$ est de type $(2, 2, 2)$ si et seulement si deux au moins  des éléments \{$(\frac{p_1}{p_2})$, $(\frac{2}{p_1})$, $(\frac{2}{p_2})$\} valent -1.
\end{proof}
 \section{\textbf{Capitulation dans les extensions $\KK_1$, $\KK_2$ et $\KK_3$}}
\subsection{\textbf{Capitulation dans les extensions $\KK_1$ et $\KK_2$}}
Commençons cette partie par rappeler quelques résultats que nous serons utiles dans la suite.\\
\indent Soient m et n deux entiers naturels différents et sans facteurs carrés, $\varepsilon_1$ (resp $\varepsilon_2$, $\varepsilon_3$) l'unité fondamentale de $\QQ(\sqrt m)$ (resp $\QQ(\sqrt n)$, $\QQ(\sqrt {mn})$). On suppose que $\varepsilon_1$, $\varepsilon_2$ et $\varepsilon_3$ sont de norme -1. On pose $K_0=\QQ(\sqrt m,\sqrt n)$ et $K=K_0(i)$.
\begin{propo}[\cite{Az-05}] \label{18}
\indent 1) Si $\varepsilon_1\varepsilon_2\varepsilon_3$ est un carré dans  $K_0$, alors $\{\varepsilon_1, \varepsilon_2,\sqrt{\varepsilon_1\varepsilon_2\varepsilon_3}\}$ est un SFU de $K_0$.\\
\indent 2) sinon \{$\varepsilon_1$, $\varepsilon_2$, $\varepsilon_3$\} est un SFU de $K_0.$
\end{propo}
\begin{propo}[\cite{Az-05}]\label{19}
On note par $Q$ l'indice d'unité de $K$, alors on a:\\
\indent 1) Si $\{\varepsilon_1, \varepsilon_2,\sqrt{\varepsilon_1\varepsilon_2\varepsilon_3}\}$ est un SFU de $K_0$, alors $Q=1$.\\
\indent 2) Si \{$\varepsilon_1$, $\varepsilon_2$, $\varepsilon_3$\} est un SFU de $K_0$, alors $Q=2$ ssi $2\varepsilon_1\varepsilon_2\varepsilon_3$ est un carré dans $K_0$.\\
\indent 3) Si $Q=2$, alors $\{\varepsilon_1, \varepsilon_2,\sqrt{i\varepsilon_1\varepsilon_2\varepsilon_3}\}$ est un SFU de $K.$
\end{propo}
\begin{lem}[\cite{Az-00}] \label{20}
Soient $p$ un nombre premier impair et $\varepsilon=x+y\sqrt{2p}$ l'unité fondamentale de $\QQ(\sqrt{2p})$. On suppose que $\varepsilon$ est de norme 1. Alors $x\pm1$ est un carré dans $\NN$ et $2\varepsilon$ est un carré dans $\QQ(\sqrt{2p})$.
\end{lem}
\begin{propo}\label{21}
On garde les notations et les hypothèses précédentes et désignant par $p$ et $q$ deux nombres premiers congrus à $1\pmod 4$. Posons $m=p$ et $n=2q$. Alors  $\{\varepsilon_1, \varepsilon_2,\sqrt{\varepsilon_1\varepsilon_2\varepsilon_3}\}$ ou $\{\varepsilon_1, \varepsilon_2,\sqrt{i\varepsilon_1\varepsilon_2\varepsilon_3}\}$ est SFU de $K$.
\end{propo}
\begin{proof}
Soient $\varepsilon_1$ (resp $\varepsilon_2$, $\varepsilon_3$) l'unité fondamentale de $\QQ(\sqrt{p})$ (resp $\QQ(\sqrt{2q})$, $\QQ(\sqrt{2pq})$), comme $p$ et $q$ sont congrus à 1 $\pmod 4$, alors ils existent $\pi_1$, $\pi_2$, $\pi_3$ et $\pi_4$ dans $\ZZ[i]$ tels que $p=\pi_1\pi_2$, $q=\pi_3\pi_4$, $\overline{\pi_1}=\pi_2$ et $\overline{\pi_3}=\pi_4$ (le conjugué complexe). Posons $\varepsilon_1=x+y\sqrt {p}$ l'unité fondamentale de $\QQ(\sqrt p)$, où $x$ et $y$ sont des entiers ou des semi-entiers.\\
 \indent (a) Supposons que $x$ et $y$ sont des entiers. Comme la norme de  $\varepsilon_1$ est égale $-1$, alors $(x-i)(x+i)=py^2$ et le pgcd de $x-i$ et $x+i$ est un diviseur de 2, donc il existent $y_1$ et $y_2$ dans $\ZZ[i]$ tels que $y=y_1y_2$ et\\
$\left\{\begin{array}{ll}
    x+i =y_1^2\pi_1\\
    x-i =y_2^2\pi_2
    \end{array}\right.
\text{   ou   } \left\{\begin{array}{ll}
    x+i =iy_1^2\pi_1\\
    x-i =-iy_2^2\pi_2.
    \end{array}\right.$\\
 Par suite $2x=y_1^2\pi_1+y_2^2\pi_2$ ou $2x=iy_1^2\pi_1-iy_2^2\pi_2$, il s'en suit alors que $2\varepsilon_1=(y_1\sqrt{\pi_1}+y_2\sqrt{\pi_2})^2$ ou $2\varepsilon_1=(y_1\sqrt{i\pi_1}+y_2\sqrt{-i\pi_2})^2$ d'où $\sqrt{2\varepsilon_1}=y_1\sqrt{\pi_1}+y_2\sqrt{\pi_2}$ ou $\sqrt{\varepsilon_1}=y_1(1+i)\sqrt{\pi_1}+y_2(1-i)\sqrt{\pi_2}$, on conclut que:
   \begin{equation}\label{13}\left.
   \begin{array}{ll}\sqrt{2\pi_1\varepsilon_1} = y_1\pi_1+y_2\sqrt{p},\ et  \\
    \sqrt{2\pi_2\varepsilon_1}=y_1\sqrt{p}+y_2\pi_2,
    \ ou \\
     \sqrt{\pi_1\varepsilon_1}=y_1(1+i)\pi_1+\\y_2(1-i)\sqrt{p}\ \ et\\
      \sqrt{\pi_2\varepsilon_1}=y_1(1+i)\sqrt{p}+\\y_2(1-i)\pi_2.
     \end{array}
   \right\}
\end{equation}
 \indent (b) Soit $\varepsilon_1=\frac{1}{2}(x+y\sqrt {p})$ avec $x$ et $y$ sont des entiers impairs, alors on a $(x-2i)(x+2i)=\pi_1\pi_2y^2$, on procède comme précédemment on trouve les mêmes résultats.\\
 \indent c) Posons de même $\varepsilon_2=\alpha+\beta\sqrt{2q}$ avec $\alpha$ et $\beta$ deux entiers, on trouve aussi que:\\
 \begin{equation}\label{14}
 \left.
 \begin{array}{ll}
\sqrt{2(1+i)\pi_3\varepsilon_2}=\beta_1(1+i)\pi_3+\\\beta_2\sqrt{2q}\ \text{  et  } \\ \sqrt{2(1-i)\pi_4\varepsilon_2}=\beta_1\sqrt{2q}+\\\beta_2(1-i)\pi_4  \text{  ou  }\\
\sqrt{(1+i)\pi_3\varepsilon_2}=\frac{1}{2}(\beta_1(1+i)(1\pm \\i)\pi_3+\beta_2(1\mp i)\sqrt{2q})  \text{  et  }\\ \sqrt{(1-i)\pi_4\varepsilon_2}=\frac{1}{2}(\beta_1(1\pm \\i)\sqrt{2q}+\beta_2(1-i)(1\mp i)\pi_4).
 \end{array}\right\}
\end{equation}
 \indent (d) On applique le même argument à  $\varepsilon_3$ on trouve aussi que:\\
  \begin{equation}\label{15}
 \left.
 \begin{array}{ll}
    \sqrt{2(1+i)\pi_1\pi_3\varepsilon_3}=y_1(1+i)\pi_1\pi_3\\+y_2\sqrt{2pq} \text{  et  }\\ \sqrt{2(1-i)\pi_2\pi_4\varepsilon_3}=y_1\sqrt{2pq}\\+y_2(1-i)\pi_2\pi_4\text{ ou }\\
    \sqrt{(1+i)\pi_1\pi_3\varepsilon_3}=\frac{1}{2}(y_1(1+i)\\(1\pm i)\pi_1\pi_3+y_2(1\mp i)\sqrt{2pq})
    \text{ et }\\ \sqrt{(1-i)\pi_2\pi_4\varepsilon_3}=\frac{1}{2}(y_1(1\pm i)\\\sqrt{2pq}+y_2(1\mp i)(1-i)\pi_2\pi_4).
 \end{array}\right\}
\end{equation}

En multipliant les résultats des égalités (\ref{13}), (\ref{14}) et (\ref{15}), on trouve que $\sqrt{\varepsilon_1\varepsilon_2\varepsilon_3}\in K_0$ ou bien $\sqrt{2\varepsilon_1\varepsilon_2\varepsilon_3}\in K_0$, remarquons enfin que  $\sqrt{\varepsilon_1\varepsilon_2\varepsilon_3}$ et $\sqrt{2\varepsilon_1\varepsilon_2\varepsilon_3}$ ne sont pas tous les deux dans $K_0$, sinon on aura $\sqrt 2\in K_0$ ce qui est faux. La suite est une déduction simple des propositions (\ref{18}) et (\ref{19}).
\end{proof}
\begin{propo}\label{22}
Gardons les hypothèses de la proposition (\ref{21}) et supposons que la norme de $\varepsilon_2$ est égale à $1$.  Alors  le SFU de $K_0$ est \{$\varepsilon_1$, $\varepsilon_2$, $\varepsilon_3$\} et celui de $K$ est \{$\varepsilon_1$, $\sqrt{i\varepsilon_2}$, $\varepsilon_3$\}.\\
\end{propo}
\begin{proof}
 Comme la norme de $\varepsilon_2$ est égale à 1, alors d'après le lemme (\ref{20}), $x\pm1$ est un carré dans $\NN$ et $2\varepsilon_2$ est un carré dans $\QQ(\sqrt{2q})$ ($\varepsilon_2=x+y\sqrt{2q}$), de plus
  $\varepsilon_2$ n'est pas un carré dans $K_0$, sinon on trouve que $\sqrt 2\in K_0$, ce qui est absurde.
  Comme $\varepsilon_1$ et $\varepsilon_2$ ont pour norme $-1$, alors elles ne sont pas des carrés dans $K_0$.
   De même $\varepsilon_1\varepsilon_2$, $\varepsilon_1\varepsilon_3$, $\varepsilon_2\varepsilon_3$ et $\varepsilon_1\varepsilon_2\varepsilon_3$ ne sont pas des carrés dans $K_0$,  sinon on trouve que $i\in K_0$, ce qui est absurde.
Donc \{$\varepsilon_1$, $\varepsilon_2$, $\varepsilon_3$\} est un SFU de $K_0$ et comme $2\varepsilon_2$ est un carré dans $K_0$, alors d'après \cite{Az-05},  \{$\varepsilon_1$, $\sqrt{i\varepsilon_2}$, $\varepsilon_3$\} est un SFU de $K$.
\end{proof}
\begin{them}\label{23}
Soit le corps $\kk=\QQ(\sqrt{2p_1p_2}, i)$ avec $p_1$ et $p_2$ deux nombres premiers congrus à $ 1\pmod 4$
et  au moins deux des éléments \{$(\frac{p_1}{p_2})$, $(\frac{2}{p_1})$, $(\frac{2}{p_2})$\} valent $-1$ et soit $\mathbf{C}_{\mathds{k},2}$ le $2$-groupe de classes de $\kk$. Posons $\KK_1=\kk(\sqrt {p_1})=\QQ(\sqrt{p_1}, \sqrt{2p_2}, i)$, $\KK_2=\kk(\sqrt {p_2})=\QQ(\sqrt{p_2}, \sqrt{2p_1}, i)$, $p_1=e^2+(2f)^2$ et $p_1=\pi_1\pi_2=(e+2if)(e-2if)$. Soient $\mathcal{H}_0$, $\mathcal{H}_1$ et $\mathcal{H}_2$ les idéaux de $\kk$ au dessus de $1+i$, $\pi_1$ et $\pi_2$ respectivement, alors exactement quatre classes de $\mathbf{C}_{\mathds{k},2}$ capitulent dans $\KK_1$ et $\KK_2$, de plus  $kerj_{\KK_1}=\langle[\mathcal{H}_1], [\mathcal{H}_2]\rangle$ et $kerj_{\KK_2}=\langle[\mathcal{H}_0\mathcal{H}_1], [\mathcal{H}_0\mathcal{H}_2]\rangle$.
\end{them}
\begin{proof}  Comme $p_1\equiv p_2\equiv1\pmod 4$, alors la norme de l'unité fondamentale de $\QQ(\sqrt{p_1})$
(resp $\QQ(\sqrt{p_2})$) est égale à $-1$, et d'après le lemme (\ref{4}) la norme de l'unité fondamentale de
 $\QQ(\sqrt{2p_1p_2})$ est égale à $-1$. Donc d'après les proposition (\ref{21}) et (\ref{22}), $\KK_j$,  où
 $j\in\{1,2\}$, a pour SFU l'un des  trois systèmes  $\{\varepsilon_1, \sqrt{i\varepsilon_2}, \varepsilon_3\}$ ou $\{\varepsilon_1, \varepsilon_2,\sqrt{\varepsilon_1\varepsilon_2\varepsilon_3}\}$ ou $\{\varepsilon_1, \varepsilon_2,\sqrt{i\varepsilon_1\varepsilon_2\varepsilon_3}\}$.
 On note par $E_{\KK_j}$ le groupe des unités de $\KK_j$, alors $E_{\KK_j}=\langle i, \varepsilon_1, \sqrt{i\varepsilon_2}, \varepsilon_3\rangle$ ou  $E_{\KK_j}=\langle i, \varepsilon_1, \varepsilon_2,\sqrt{\varepsilon_1\varepsilon_2\varepsilon_3}\rangle$ ou $E_{\KK_j}=\langle i,\varepsilon_1, \varepsilon_2,\sqrt{i\varepsilon_1\varepsilon_2\varepsilon_3}\rangle$. Donc $N_{\KK_j/\kk}(E_{\KK_j})=\langle i, \varepsilon_3^2\rangle$ (dans ce cas la norme de l'unité fondamentale de $\QQ(\sqrt{2p_2})$ (resp $\QQ(\sqrt{2p_1})$)  est 1) ou $N_{\KK_j/\kk}(E_{\KK_j})=\langle -1, \varepsilon_3\rangle$  ou $N_{\KK_j/\kk}(E_{\KK_j})=\langle -1, i\varepsilon_3\rangle$, d'autre part on sait d'après \cite{Az-05} que le groupe d'unités de $\kk$ est $E_{\kk}=\langle i, \varepsilon_3\rangle$, d'où $[E_{\kk}:N_{\KK_j/\kk}(E_{\KK_j})]=2$, donc d'après le théorème (\ref{17}), le nombre des classes qui capitulent dans $\KK_j$ est égale à 4.\\
 \indent Dans la suite on va déterminer ces classes dans chaque cas:\\
\indent (i) $kerj_{\KK_1}=\langle[\mathcal{H}_1], [\mathcal{H}_2]\rangle$. En effet \\
 $\bullet $ Comme $(\mathcal{H}_1\mathcal{H}_2)^2=(\pi_1\pi_2)=(p_1)$, alors $\mathcal{H}_1\mathcal{H}_2=(\sqrt{p_1})$, donc l'idéal $\mathcal{H}_1\mathcal{H}_2$ capitule dans $\KK_1$.\\
  $\bullet $  Montrons que $[\mathcal{H}_2]$ la classe de $\mathcal{H}_2$ capitule dans $\KK_1$. D'après la démonstration de la proposition (\ref{21}), on a: $\sqrt{2\pi_2\varepsilon_1}\in\KK_1$ ou $\sqrt{\pi_2\varepsilon_1}\in\KK_1$, où $\varepsilon_1$ est l'unité fondamentale de $\QQ(\sqrt{p_1})$, alors ils existent $\alpha$ et $\beta$ dans $\KK_1$ tels que $2\pi_2\varepsilon_1=\alpha^2$ ou $\pi_2\varepsilon_1=\beta^2$, donc $(2\pi_2)=(\alpha^2)$ ou $(\pi_2)=(\beta^2)$, or $\mathcal{H}_0^2=(1+i)$ et  $\mathcal{H}_2^2=(\pi_2)$, alors $(1+i)\mathcal{H}_2=(\alpha)$ ou  $\mathcal{H}_2=(\beta)$, par suite  $\mathcal{H}_2=(\frac{\alpha}{1+i})$ ou  $\mathcal{H}_2=(\beta)$, donc $\mathcal{H}_2$ capitule dans $\KK_1$.\\
\indent   Comme $[\mathcal{H}_1\mathcal{H}_2]$ et $[\mathcal{H}_2]$ sont principaux dans $\KK_1$, alors $[\mathcal{H}_1\mathcal{H}_2\mathcal{H}_2]=[\mathcal{H}_1]$ est principal dans $\KK_1$. Par suite $kerj_{\KK_1}=\langle[\mathcal{H}_1], [\mathcal{H}_2]\rangle$.\\
\indent (ii) Montrons que $kerj_{\KK_2}=\langle[\mathcal{H}_0\mathcal{H}_1], [\mathcal{H}_0\mathcal{H}_2]\rangle$, on sait que $\KK_2=\kk(\sqrt {p_2})=\QQ(\sqrt{p_2}, \sqrt{2p_1}, i)$, $p_1=\pi_1\pi_2$   et $p_2=\pi_3\pi_4$, désignons par $\varepsilon_1$ (resp $\varepsilon_3$) l'unité fondamentale de $\QQ(\sqrt{p_2})$ (resp $\QQ(\sqrt{2p_1p_2})$), alors en multipliant les résultats des égalités (\ref{13}) et (\ref{15}) on trouve que  $\sqrt{(1+i)\pi_1\varepsilon_1\varepsilon_3}\in\KK_2$  et  $\sqrt{(1-i)\pi_2\varepsilon_1\varepsilon_3}\in\KK_2$, donc il existe $\alpha$ et $\beta$ dans $\KK_2$ tels que $\mathcal{H}_0^2\mathcal{H}_1^2=(\alpha^2)$ et  $\mathcal{H}_0^2\mathcal{H}_2^2=(\beta^2)$, alors
 $\mathcal{H}_0\mathcal{H}_1=(\alpha)$ et  $\mathcal{H}_0\mathcal{H}_2=(\beta)$, d'où $\mathcal{H}_0\mathcal{H}_1$ et  $\mathcal{H}_0\mathcal{H}_2$ capitulent dans $\KK_2$
\end{proof}
\subsection{\textbf{Capitulation dans l'extension $\KK_3$}}
Soient  $p_1$ et $p_2$  deux nombres premiers congrus à $1 \pmod 4$, $\kk=\QQ(\sqrt{2p_1p_2},i)$, $\KK_3=\kk(\sqrt{2})=\QQ(\sqrt 2, \sqrt{p_1p_2}, i)$, $\KK_3^+=\QQ(\sqrt 2,\sqrt{p_1p_2})$, $\varepsilon_1$ (resp. $\varepsilon_2$, $\varepsilon_3$ ) l'unité fondamentale de $k_1=\QQ(\sqrt 2)$ (resp. $k_2=\QQ(\sqrt{p_1p_2})$, $k_3=\QQ(\sqrt{2p_1p_2})$), $I=\{0, 1\}$ et $E_\kk$ (resp. $E_{\KK_3})$ le groupe d'unité de $\kk$ ( resp. $\KK_3$). On sait que $\varepsilon_1$ et $\varepsilon_3$ sont de normes $-1$, par contre $\varepsilon_2$ est de norme $\pm1$.
\begin{lem}[\cite{Az-00}]\label{24}
Soient $d$ un entier relatif sans facteur carré et $\varepsilon=x+y\sqrt d$ l'unité fondamentale de $\QQ(\sqrt d)$ où x et y sont des entiers ou bien  des semi-entiers. Si $\varepsilon$ est de norme $1$, alors $2(x\pm1)$ et $2d(x\pm1)$ ne sont pas des carrés dans $\QQ$.
\end{lem}
\begin{lem}\label{28}
 Si la norme de $\varepsilon_2=x+y\sqrt{p_1p_2}$ est égale à 1, alors $x\pm1$ n'est pas un carré dans $\NN.$
\end{lem}
\begin{proof}
 Comme $p_1p_2\equiv 1 \pmod 4$, alors l'indice d'unité de $k_2(i)$ est égal à 1,  donc $2\varepsilon_2$ n'est pas un carré dans $k_2$, ceci implique que $x\pm1$ n'est pas un carré dans $\NN$.
\end{proof}
\begin{propo}\label{29}
Supposons que  la norme de $\varepsilon_2$, l'unité fondamental de $k_2=\QQ(\sqrt{p_1p_2})$, est égale à $1$, alors $\{\varepsilon_1,  \varepsilon_2,  \varepsilon_3 \}$ est un SFU de $\KK_3^+$ et de $\KK_3.$
\end{propo}
\begin{proof}
Puisque $\varepsilon_1$ et $\varepsilon_3$ sont de normes $-1$, alors elles ne sont pas des carrés dans $\KK_3^+$, de même $\varepsilon_1\varepsilon_2$, $\varepsilon_1\varepsilon_3$, $\varepsilon_2\varepsilon_3$ et $\varepsilon_1\varepsilon_2\varepsilon_3$  ne sont pas des carrés dans $\KK_3^+$, sinon en prenant une norme convenable on trouve que $i\in \KK_3^+$, ce qui est faux. De plus
 $(2+\sqrt2)\varepsilon_1^i\varepsilon_2^j\varepsilon_3^k$ ne peut pas être un carré dans $\KK_3^+$, pour tout i, j et k de $I$, sinon il existe $\alpha\in\KK_3^+$ tel que $\alpha^2=(2+\sqrt 2)\varepsilon_1^i\varepsilon_2^j\varepsilon_3^k$, alors $(N_{\KK_3^+/k_2}(\alpha))^2=2(-1)^{i+k}\varepsilon_2^{2j}$, ceci implique que $\sqrt2\in\QQ(\sqrt{p_1p_2})$, ce qui est absurde.\\
 \indent  (1) Soit $\varepsilon_2=x+y\sqrt{p_1p_2}$, avec x et y deux entiers, alors $x^2-1=y^2p_1p_2$. d'après le lemme (\ref{24}), $2(x\pm1)$ et $2p_1p_2(x\pm1)$ ne sont pas des  carrés dans $\NN$, ainsi on a:\\
 \indent (i) Si $p_1p_2(x\pm1)$ est un carré dans $\NN$, alors il existe $(y_1, y_2)\in\ZZ^2$ tel que
$\left\{
 \begin{array}{ll}
 x\mp1=y_1^2,\\
 x\pm1=y_2^2p_1p_2.
 \end{array}\right.$\\
 d'où $x\mp1$ est un carré dans $\NN$, et ceci contredit le lemme (\ref{28}).\\
\indent (ii) Soit $p_1(x\pm1)$ ou $2p_1(x\pm1)$ un carré dans $\NN$. Si $p_1(x\pm1)$ est  un carré dans $\NN$, alors  il existe $(y_1, y_2)\in\ZZ^2$ tel que
$\left\{
 \begin{array}{ll}
 x\pm1=p_1y_1^2,\\
 x\mp1=p_2y_2^2.
 \end{array}\right.$\\
$\sqrt{\varepsilon_2}=\frac{1}{2}(y_1\sqrt{2p_1}+y_2\sqrt{2p_2})$, donc $\sqrt{\varepsilon_2}\not\in\KK_3^+$, $\sqrt{p_1\varepsilon_2}\in\KK_3^+$ et $\sqrt{p_2\varepsilon_2}\in\KK_3^+$. On trouve les mêmes résultats pour le cas: $2p_1(x\pm1)$.\\
 \indent (2) Soit $\varepsilon_2=\frac{1}{2}(x+y\sqrt{p_1p_2})$, où x et y sont deux entiers impairs, comme $\varepsilon_2$ est de norme $1$, alors  $(x-2)(x+2)=y^2p_1p_2$, $(x+2)-(x-2)=4$ et le pgcd de $x+2$ et $x-2$ divise $4$, or  $x+2$ et $x-2$ sont impairs, donc leur pgcd est égal à $1$.\\
 \indent (i) Si $p_1p_2(x\pm2)$ un carré dans $\NN$, alors il existe $(y_1, y_2)\in\ZZ^2$ tel que
$\left\{
 \begin{array}{ll}
 x\pm2=p_1p_2y_1^2,\\
 x\mp2=y_2^2.
 \end{array}\right.$\\
 D'où $\sqrt{\varepsilon_2}=\frac{1}{2}(y_1\sqrt{p_1p_2}+y_2)$, donc $\sqrt{\varepsilon_2}\in k_2=\QQ(\sqrt{p_1p_2})$, ce qui est absurde puisque $\varepsilon_2$ est l'unité fondamentale de $k_2$. Donc ce cas ne se présente pas. \\
  \indent (ii) Si $p_1(x\pm2)$ un carré dans $\NN$, alors il existe $(y_1, y_2)\in\ZZ^2$ tel que
$\left\{
 \begin{array}{ll}
 x\pm2=p_1y_1^2,\\
 x\mp2=p_2y_2^2.
 \end{array}\right.$\\
 Donc on trouve que  $\sqrt{\varepsilon_2}=\frac{1}{2}(y_1\sqrt{2p_1}+y_2\sqrt{2p_2})$, par suite $\sqrt{\varepsilon_2}\not\in\KK_3^+$, $\sqrt{p_1\varepsilon_2}\in\KK_3^+$ et $\sqrt{p_2\varepsilon_2}\in\KK_3^+$.\\
  De tous ces résultats  en déduit que $\{\varepsilon_1,  \varepsilon_2,  \varepsilon_3 \}$ est SFU de $\KK_3^+$ et
 d'après \cite{Az-05}, $\KK_3$ a le même système fondamental d'unités que $\KK_3^+$.
\end{proof}
\begin{propo}\label{34}
On suppose que la norme de $\varepsilon_2$ est égale à $-1$,  alors on a:\\
\indent (i) Si $\varepsilon_1\varepsilon_2\varepsilon_3$ est un carré dans $\KK_3^+$, alors
 $\{ \varepsilon_1, \varepsilon_2, \sqrt{\varepsilon_1\varepsilon_2\varepsilon_3}\}$ est un SFU de $\KK_3^+$ et de $\KK_3$.\\
 \indent (ii) Sinon $\{ \varepsilon_1, \varepsilon_2, \varepsilon_3\}$ est un SFU de $\KK_3^+$ et de $\KK_3$.
\end{propo}
\begin{proof}
D'après la proposition (\ref{18}), $\KK_3^+$ a pour SFU $\{ \varepsilon_1, \varepsilon_2, \sqrt{\varepsilon_1\varepsilon_2\varepsilon_3}\}$ ou $\{\varepsilon_1, \varepsilon_2, \varepsilon_3\}$ suivant que $\varepsilon_1\varepsilon_2\varepsilon_3$ est un carré ou non dans $\KK_3^+$, de plus d'après \cite{Az-05}, $\KK_3^+$ et  $\KK_3$ ont même SFU.
\end{proof}
\brem\label{38}
 On procède comme dans la démonstration de la proposition (\ref{21}), pour prouver ce qui suit:\\
 a) $\varepsilon_2=x+y\sqrt{p_1p_2}$
  \begin{equation}\label{32}\left.
   \begin{aligned}\sqrt{\varepsilon_2}& = y_1\sqrt{\pi_1\pi_3}+y_2\sqrt{\pi_2\pi_4},\ ou  \\
   \sqrt{\varepsilon_2}& = y_1\sqrt{\pi_1\pi_4}+y_2\sqrt{\pi_2\pi_3},\ ou
    \  \\
    \sqrt{2\varepsilon_2}& = y_1\sqrt{\pi_1\pi_3}+y_2\sqrt{\pi_2\pi_4},\ ou \\
     \sqrt{2\varepsilon_2}& = y_1\sqrt{\pi_1\pi_4}+y_2\sqrt{\pi_2\pi_3},
     \end{aligned}
   \right\}
\end{equation}
où les $y_i$ sont dans $\ZZ[i]$ ou $\frac{1}{2}\ZZ[i]$.\\
  \indent b) De même en posant $\varepsilon_3=a+b\sqrt{2p_1p_2}$, on montre  que:
    \begin{equation}\label{33}
    \left.
   \begin{aligned}
  \sqrt{\varepsilon_3} = b_1\sqrt{(1+i)\pi_1\pi_3}+\\b_2\sqrt{(1-i)\pi_2\pi_4}), ou  \\
  \sqrt{\varepsilon_3} = b_1\sqrt{(1+i)\pi_1\pi_4}+\\b_2\sqrt{(1-i)\pi_2\pi_3}), ou
      \\
    \sqrt{2\varepsilon_3} = b_1\sqrt{(1+i)\pi_1\pi_3}+\\b_2\sqrt{(1-i)\pi_2\pi_4}), ou \\
    \sqrt{2\varepsilon_3} = b_1\sqrt{(1+i)\pi_1\pi_3}+\\b_2\sqrt{(1-i)\pi_2\pi_4}),
     \end{aligned}
   \right\}
\end{equation}
où les $b_i$ sont dans $\ZZ[i]$ ou $\frac{1}{2}\ZZ[i]$.
  Remarquons en fin que:
 \begin{equation}\label{37}
  \sqrt{2\varepsilon_1}=\sqrt{1+i}+\sqrt{1-i}.\end{equation}
Alors en multipliant les résultats des égalités (\ref{32}), (\ref{33}) et (\ref{37}) on trouve, dans le cas où  $\varepsilon_1\varepsilon_2\varepsilon_3$ n'est pas un carré dans $\QQ(\sqrt2, \sqrt{p_1p_2})$,  que  $\sqrt{\varepsilon_1\varepsilon_2\varepsilon_3}=\alpha\sqrt{p_1}+\beta\sqrt{p_2}+\gamma\sqrt{2p_1}+\delta\sqrt{2p_2}$, où $\alpha$, $\beta$, $\gamma$ et $\delta$ sont dans $\QQ$.
\erem
\begin{them}\label{30}
~Soient le corps $\kk=\QQ(\sqrt{2p_1p_2}, i)$, avec $p_1$ et $p_2$ deux nombres premiers congrus à $ 1\pmod 4$
et  au moins deux des éléments \{$(\frac{p_1}{p_2})$, $(\frac{2}{p_1})$, $(\frac{2}{p_2})$\} valent $-1$ et $\mathbf{C}_{\mathds{k},2}$ le $2$-groupe de classes de $\kk$. Posons $\KK_3=\kk(\sqrt {2})=\QQ(\sqrt{2}, \sqrt{p_1p_2}, i)$, $p_1=e^2+(2f)^2=\pi_1\pi_2=(e+2if)(e-2if)$. Soient $\mathcal{H}_0$, $\mathcal{H}_1$ et $\mathcal{H}_2$ les idéaux de $\kk$ au dessus de $1+i$, $\pi_1$ et $\pi_2$ respectivement. Notons par $Q$ l'indice des unités de $\QQ(\sqrt2, \sqrt{p_1p_2})$, d'où on a:
 \begin{enumerate}
   \item Si $Q=1$, alors  quatre classes de $\mathbf{C}_{\mathds{k},2}$ capitulent dans $\KK_3$ et $kerj_{\KK_3}=\langle[\mathcal{H}_0],[\mathcal{H}_1\mathcal{H}_2]\rangle$.
   \item Si $Q=2$, alors  deux classes de $\mathbf{C}_{\mathds{k},2}$ capitulent dans $\KK_3$ et $kerj_{\KK_3}=\langle[\mathcal{H}_0]\rangle$.
 \end{enumerate}
\end{them}
\begin{proof}Remarquons d'abord que $\sqrt i=\frac{1+i}{\sqrt2}\in\KK_3$.\\
$(1)$ Si $Q=1$, alors la norme de $\varepsilon_2$ est égale à $1$ ou la norme de $\varepsilon_2$ est égale à $-1$ et $\varepsilon_1\varepsilon_2\varepsilon_3$ n'est pas un carré dans $\QQ(\sqrt2, \sqrt{p_1p_2})$,  donc  d'après les propositions (\ref{29}) et (\ref{34})
 $E_{\KK_3}=\langle \sqrt i, \varepsilon_1, \varepsilon_2, \varepsilon_3\rangle$, ce qui implique que $N_{\KK_3/\kk}(E_{\KK_3})=\langle i, \varepsilon_3^2\rangle$, or $E_{\kk}=\langle  i, \varepsilon_3\rangle$, d'où $[E_\kk:N_{\KK_3/\kk}(E_{\KK_3})]=2$, alors d'après le théorème (\ref{17}), quatre classes  capitulent dans $\KK_3$.\\
\indent  Montrons que $[\mathcal{H}_0]$ et $[\mathcal{H}_1\mathcal{H}_2]$ sont principaux dans  $\KK_3$.\\
On remarque que $ \sqrt{(1+i)\varepsilon_1}=\frac{1}{2}[2+(1+i)\sqrt 2]\in\KK_3$, alors il existe un $\alpha\in\KK_3$ tel que $(1+i)\varepsilon_1=\alpha^2$, et comme $\mathcal{H}_0^2=(1+i)$ et $\varepsilon_1$ est une unité de $\KK_3$, alors $\mathcal{H}_0^2=(\alpha^2)$, d'où $\mathcal{H}_0=(\alpha)$.\\
 Si la norme de $\varepsilon_2$ est égale à $1$, alors  on a, d'après la démonstration de la proposition (\ref{29}), que $\sqrt{p_1\varepsilon_2}\in\KK_3$, donc il existe un $\beta\in\KK_3$ tel que $p_1\varepsilon_2=\beta^2$, d'où $(p_1)=(\beta^2)$, or  $(\mathcal{H}_1\mathcal{H}_2)^2=(p_1)$, donc $(\mathcal{H}_1\mathcal{H}_2)^2=(\beta^2)$, on conclut que  $\mathcal{H}_1\mathcal{H}_2=(\beta)$.\\
 Si la norme de $\varepsilon_2$ est égale à $-1$ et $\varepsilon_1\varepsilon_2\varepsilon_3$ n'est pas un carré dans $\QQ(\sqrt2, \sqrt{p_1p_2})$,  alors  d'après la remarque (\ref{38}), $\sqrt{p_1\varepsilon_1\varepsilon_2\varepsilon_3}\in\KK_3$, ceci implique $\mathcal{H}_1\mathcal{H}_2$ capitule dans $\KK_3$. Donc $kerj_{\KK_3}=\langle [\mathcal{H}_0], [\mathcal{H}_1\mathcal{H}_2]\rangle$.\\
 \indent (2) Si $Q=2$, alors la norme de $\varepsilon_2$ est égale à $-1$ et $\varepsilon_1\varepsilon_2\varepsilon_3$ est  un carré dans $\QQ(\sqrt2, \sqrt{p_1p_2})$, par suite d'après la proposition (\ref{34}) on a: $E_{\KK_3}=\langle\sqrt{i}, \varepsilon_1, \varepsilon_2, \sqrt{\varepsilon_1\varepsilon_2\varepsilon_3}\rangle$, d'où  $N_{\KK_3/\kk}(E_{\KK_3})=\langle{i}, \varepsilon_3\rangle$, et comme $E_{\kk}= \langle{i}, \varepsilon_3\rangle$, alors $[E_{\kk}:N_{\KK_3/\kk}(E_{\KK_3})]=1$, donc le théorème (\ref{17}) implique que deux classes  capitulent dans $\KK_3$. Et d'après la démonstration faite dans (1), on a: $kerj_{\KK_3}=\langle [\mathcal{H}_0]\rangle$. Et ceci achève la démonstration du théorème.
\end{proof}
Des théorèmes (\ref{23}) et (\ref{30}) découle le théorème:
\begin{them}
Soit le corps $\kk=\QQ(\sqrt{2p_1p_2}, i)$ avec $p_1$ et $p_2$ deux nombres premiers congrus à $ 1\pmod 4$
et  au moins deux des éléments \{$(\frac{p_1}{p_2})$, $(\frac{2}{p_1})$, $(\frac{2}{p_2})$\} valent $-1$ et soit $\mathbf{C}_{\mathds{k},2}$ le $2$-groupe de classes de $\kk$. Posons $\kk^*$ le corps de genres de $\kk$, $p_1=e^2+(2f)^2$ et $p_1=\pi_1\pi_2=(e+2if)(e-2if)$. Soient $\mathcal{H}_0$, $\mathcal{H}_1$ et $\mathcal{H}_2$ les idéaux de $\kk$ au dessus de $1+i$, $\pi_1$ et $\pi_2$ respectivement. Alors toutes les classes capitulent dans $\kk^*$ c-à-d: $kerj_{\kk^*}=\mathbf{C}_{\mathds{k},2}$.
\end{them}

\end{document}